\documentclass{amsart}

\usepackage{geometry}
\usepackage[utf8]{inputenc} 
\usepackage[T1]{fontenc} 
\usepackage[english]{babel}
\usepackage{amssymb}
\usepackage{epsfig}
\usepackage{enumitem}

\theoremstyle{plain} 
\newtheorem{theorem}{Theorem} 
 
\newtheorem{lemma}[theorem]{Lemma}

\newtheorem{remark}[theorem]{Remark} 
\newtheorem{definition}[theorem]{Definition}

\usepackage{color}
	
\newcommand{\br}{\mathbf r}
\newcommand{\bnu}{{\boldsymbol \nu}}

\definecolor{sps}{rgb}{0.2,0.7,0.2}

\newcommand{\cone}{1}

\DeclareMathOperator{\mre}{Re}
\DeclareMathOperator{\ess}{ess}

\begin{document} 
\title[Infinitely many embedded eigenvalues for the Neumann-Poincar\'e operator in 3D]{Infinitely many embedded eigenvalues\\ for the Neumann-Poincar\'e operator in 3D}
\date{\today} 

\author{Wei Li} 
\address{Department of Mathematical Sciences,
	DePaul University, Chicago, IL, USA} 
\email{wei.li@depaul.edu}
\author{\,Karl-Mikael Perfekt}
\address{Department of Mathematical Sciences, Norwegian University of Science and Technology (NTNU), NO-7491 Trondheim, Norway}
\email{karl-mikael.perfekt@ntnu.no}
\author{\,Stephen P. Shipman} 
\address{Department of Mathematics
	Louisiana State University, Baton Rouge, LA, USA} 
\email{shipman@lsu.edu}


\begin{abstract}
This article constructs a surface whose Neumann-Poincar\'e (NP) integral operator has infinitely many eigenvalues embedded in its essential spectrum. The surface is a sphere perturbed by smoothly attaching a conical singularity, which imparts essential spectrum.
Rotational symmetry allows a decomposition of the operator into Fourier components.
Eigenvalues of infinitely many Fourier components are constructed so that they lie within the essential spectrum of other Fourier components and thus within the essential spectrum of the full NP operator.
The proof requires the perturbation to be sufficiently small, with controlled curvature, and the conical singularity to be sufficiently flat.
\end{abstract}
\maketitle

\section{Introduction}

Let $\Gamma \subset \mathbb R^3$ be a connected Lipschitz surface with surface measure $d\sigma$, enclosing a bounded open domain. The adjoint of the Neumann-Poincar\'e operator is defined by
\begin{equation*}
K_{\Gamma} f(\br) = \frac{1}{2\pi}\int_\Gamma K_{\Gamma}(\br,\br') f(\br') \, d\sigma(\br'), \quad \br \in \Gamma,
\end{equation*}
where the kernel is
\begin{equation}\label{eq:kerK}
K_{\Gamma}(\br,\br') = \frac{\langle \br - \br', \bnu_{\br} \rangle}{|\br - \br'|^3}.
\end{equation}
The adjoint operator $K_\Gamma^\ast$ is the direct (principal) value of the double-layer potential on $\Gamma$.
The single-layer operator is defined~by
\begin{equation*}
S_{\Gamma} f(\br) = \frac{1}{2\pi} \int_\Gamma S_{\Gamma}(\br,\br') f(\br') \, d\sigma(\br'), \quad \br \in \Gamma,
\end{equation*}
where the kernel is
\begin{equation*}
S_{\Gamma}(\br,\br') =\frac{1}{|\br - \br'|}.
\end{equation*}

While the kernel $K_{\Gamma}(\br,\br')$ is symmetric only if $\Gamma$ is a sphere, there is a natural framework that uncovers the intrinsic symmetry of the operator $K_\Gamma$ \cite{KhavinsonPutinarShapiro2007, PerfektPutinar2014}.
Consider the inner product
\begin{equation*}
\langle f, g \rangle_{S_{\Gamma}}  := \langle f, S_{\Gamma} g \rangle_{L^2(\Gamma)},
\end{equation*}
with corresponding norm $\|f\|_{S_{\Gamma}} = \sqrt{\langle f, f \rangle_{S_{\Gamma}}}$\,. 
The energy space $\mathcal{E} = \mathcal{E}(\Gamma)$ is the completion of $L^2(\Gamma)$ under the $\|\!\cdot\!\|_{S_{\Gamma}}$-norm. Then $K_\Gamma \colon \mathcal{E} \to \mathcal{E}$ is self-adjoint. For a fixed surface $\Gamma$, the space $\mathcal{E}$ and the Sobolev space $H^{-1/2}(\Gamma)$ consist of the same distributions and the norms $\|\!\cdot\!\|_{S_{\Gamma}}$ and $\|\!\cdot\!\|_{H^{-1/2}}$ are equivalent.

The main goal of the present work is to provide a surface $\Gamma$ with a conical point such that $K_\Gamma \colon \mathcal{E} \to \mathcal{E}$ has infinitely many eigenvalues embedded in the essential spectrum arising from the conical singularity. An equivalent formulation is that the Neumann--Poincar\'e operator $K_\Gamma^\ast \colon \mathcal{E}' \to \mathcal{E}'$ has infinitely many eigenvalues embedded in its essential spectrum, where $\mathcal{E}' \simeq H^{1/2}(\Gamma)$ is the natural dual space of $\mathcal{E}$. The surface will be constructed by smoothly adding a conical singularity to a sphere in such a way that rotational symmetry is preserved.

For the NP operator on certain reflectionally symmetric closed curves in $\mathbb{R}^2$ with a corner, it was proved in~\cite{LiShipman2019} that finitely many eigenvalues are embedded in the essential spectrum.  Numerical evidence of embedded eigenvalues and complex resonances had previously been demonstrated in~\cite{HelsingKangLim2016}, and numerical analysis of these phenomena appeared in \cite{Bonnet-BenHazardMonteghett2020} as applied to surface plasmons on subwavelength particles.  In 2D, the essential spectrum produced by a corner is an angle-dependent interval of absolutely continuous spectrum of multiplicity 1 \cite{BonnetierZhang2017, KangLimYu2017, Perfekt2020,PerfektPutinar2017}. Reflectional symmetry additionally induces a decomposition of the operator into even and odd components; and for the curves featured in~\cite{LiShipman2019}, non-embedded eigenvalues of one component are embedded in the continuous spectrum of the other component.

The main result of the present work is the following.

\begin{theorem}\label{thm:infemb}
Let $\Gamma_0$ be the unit sphere in $\mathbb{R}^3$. There exists a conical perturbation $\Gamma$ of $\Gamma_0$ such that $K_{\Gamma} \colon \mathcal{E} \to \mathcal{E}$ has infinitely many eigenvalues within its essential spectrum.
\end{theorem}

These are the main elements of the proof.

\begin{enumerate}[label=(\alph*)]
\item
By rotational symmetry, $K_\Gamma$ decomposes into Fourier components $K^n_\gamma$, $n \in \mathbb{Z}$.
\item
For the unit sphere $\Gamma_0$, the largest eigenvalue of $K^n_{\gamma_0}$ is $\lambda_n=1/(2|n|+1)$.
\item
Perturbing the sphere by a conical singularity imparts an interval of essential spectrum of size $O(1/|n|)$ to $K^n_\gamma$.
\item\label{flat}
If the angle of the perturbation is flat enough and $n$ is large, the essential spectrum of $K^n_\gamma$ does not overlap~$\lambda_n$.
\item\label{small}
If the perturbation is sufficiently small and shallow, with controlled curvature, then for infinitely many $n$, $\lambda_n$ gets perturbed to an eigenvalue of $K^n_\gamma$, while avoiding its essential spectrum.
\item
Infinitely many of the perturbed eigenvalues will be embedded in the essential spectrum of each of the Fourier components.
\end{enumerate}

Point \ref{small} presents the greatest challenge, and most of the analysis is dedicated to it.
To achieve control over the perturbed eigenvalues for infinitely many~$n$, one of the requirements is a uniform bound on the curvature of the perturbed profile. This is accomplished by coordinating the angle $\alpha$ of the conical perturbation and its size~$\varepsilon$. 

The background to our work is provided by \cite{HelsingPerfekt2018a}, where the essential spectrum of $K_\Gamma \colon \mathcal{E} \to \mathcal{E}$ was computed for surfaces $\Gamma$ with rotationally symmetric conical points; see Theorem~\ref{thm:essspec}. As we need to refine the associated estimates, we will revisit a number of the calculations from \cite{HelsingPerfekt2018a}. 

\smallskip
The paper is organized as follows. In Section~\ref{sec:background} we discuss the structure of $K_\Gamma$ for rotationally symmetric surfaces $\Gamma$. In Section~\ref{sec:main} we outline the proof of Theorem~\ref{thm:infemb}. Many of the details of the proof are deferred to the more technical Section~\ref{sec:est}.

\section{Preliminaries} \label{sec:background}

This section describes the following preliminary material: (i) The Fourier decomposition of the adjoint of the NP operator on rotationally symmetric surfaces and special-function representations of the associated kernels, needed for the subsequent perturbation analysis, (ii) the essential spectrum of the Fourier components for a surface with a conical singularity, and (iii) the spectrum for the sphere.

\subsection*{Notation}  For two non-negative quantities $a = a(g)$ and $b = b(g)$ that depend on the choice of an object $g\in\mathcal{G}$ within some collection of objects $\mathcal{G}$ (such as a class of admissible curves), the expression $a \lesssim b$ means that there exists a constant $C > 0$ such that $a(g) \leq Cb(g)$ for every object $g\in\mathcal{G}$. If $a \lesssim b$ and $b \lesssim a$, we write $a \approx b$.

\subsection{Rotational symmetry}

Consider a connected rotationally symmetric Lipschitz surface $\Gamma$ with parametrization
\begin{equation*}
\br(t,\theta) = (\gamma_1(t)\cos\theta, \gamma_1(t)\sin\theta, \gamma_2(t)), \quad \theta\in[0,2\pi], \quad t\in [0,1],
\end{equation*}
for two Lipschitz functions $\gamma_1$ and $\gamma_2$. We say that $\Gamma$ is generated by $\gamma(t) = (\gamma_1(t), \gamma_2(t))$. We use this parametrization to work with functions $f \colon \Gamma \rightarrow \mathbb C$ and integral kernels by writing $f(t,\theta):= f(\br(t,\theta))$ and $ K_{\Gamma}(t,\theta, t', \theta') :=  K_{\Gamma}(\br(t,\theta), \br(t', \theta'))$.  Then
\begin{equation*}
K_{\Gamma} f(t,\theta) \,=\, \frac{1}{2\pi} \int_0^{2\pi}\!\! \int_0^1  K_{\Gamma}(t,\theta, t', \theta') f(t',\theta')  \,\gamma_1(t')|\gamma'(t')|\, dt' \, d\theta'.
\end{equation*}
Let $f^n(t)$ be the $n$th Fourier coefficient of $f(t,\theta)$,
\begin{equation*}
f^n(t) \;:=\; \frac{1}{\sqrt{2\pi}} \int_0^{2\pi} e^{-i n\theta} f(t,\theta) d\theta, \quad t\in [0,1].
\end{equation*}
This Fourier transform $f(\br) \mapsto \{f^n(t)\}$ provides the decomposition
$$
L^2(\Gamma,d\sigma) \;\simeq\; \bigoplus_{n \in \mathbb{Z}} L^2\big([0,1],\, \gamma_1(t)|\gamma'(t)| \, dt\big),
$$
where $\simeq$ denotes unitary equivalence, under which 
\begin{equation*}
f(t,\theta) \;=\; \frac{1}{\sqrt{2\pi}}\sum_{n=-\infty}^{\infty} f^n(t) e^{i n\theta},
\end{equation*}
and 
\begin{equation*}
\int_\Gamma |f(\br)|^2d\sigma(\br) = \sum_{n=-\infty}^{\infty}\int_0^1 |f^n(t)|^2 \gamma_1(t)|\gamma'(t)|dt.
\end{equation*}
Due to rotational invariance, the Fourier transform of the adjoint NP kernel,
\begin{equation*}
K_{\gamma}^n(t,t') \;:=\; \frac{1}{2\pi}\int_0^{2\pi} e^{-i n\theta} K_{\Gamma}(t,\theta, t',0) d\theta, \quad t, t'\in [0,1],
\end{equation*}
realizes the decomposition of $K_\Gamma$ into integral operators acting on Fourier components,
\begin{equation*}
(K_{\Gamma} f)^n(t) \,=\,  \int_0^1 K_{\gamma}^n(t,t')f^n(t') \,\gamma_1(t')|\gamma'(t')|dt'.
\end{equation*}
Similarly, the kernels 
\begin{equation*}
S_{\gamma}^n(t,t') \,=\, \frac{1}{2\pi} \int_0^{2\pi}e^{-in\theta} S_{\Gamma}(t,\theta,t',0) \, d\theta
\end{equation*}
decompose the single-layer operator.
Denote the component operators with the same symbol as their kernels:
\begin{equation*}
(K_{\gamma}^n g)(t) =  \int_0^1 K_{\gamma}^n(t,t') g(t') \gamma_1(t') |\gamma'(t')| \, dt', \quad (S_{\gamma}^n g)(t) = \int_0^1 S_{\gamma}^n(t,t') g(t') \gamma_1(t') |\gamma'(t')| \, dt',
\end{equation*}
acting in $L^2\big([0,1],\, \gamma_1(t)|\gamma'(t)| \, dt\big)$.
Letting $\mathcal{E}_n = \mathcal{E}_n(\gamma)$ be the completion of $L^2([0,1], \gamma_1(t)|\gamma'(t)| \, dt)$ with respect to the norm $\|g\|^2_{S_{\gamma}^n} : = \langle S_{\gamma}^n g,g \rangle$ \cite[Sec.\,7.1]{HelsingPerfekt2018a}, one obtains the unitary equivalence
$$
\mathcal{E}(\Gamma) \simeq \bigoplus_{n \in \mathbb{Z}} \mathcal{E}_n(\gamma).
$$
The decompositions of both $L^2(\Gamma,d\sigma)$ and $\mathcal{E}(\Gamma)$ induce corresponding decompositions of the operators
$$
K_{\Gamma} \simeq \bigoplus_{n \in \mathbb{Z}} K_{\gamma}^n, \qquad S_{\Gamma} \simeq \bigoplus_{n \in \mathbb{Z}} S_{\gamma}^n.
$$

The kernels $K^n_\gamma(t,t')$ and $S^n_\gamma(t,t')$ can be expressed in terms of special functions; see the Appendix of~\cite{HelsingPerfekt2018a} for more details. For $n \geq 0$,
\begin{equation} \label{eq:modalformula} 
K^n_\gamma(t, t') = \frac{1}{\sqrt{2\pi^3\gamma_1(t)\gamma_1(t')}} 
\left[\frac{\gamma_2'(t)}{2\gamma_1(t)|\gamma'(t)|}
\left(\mathfrak{Q}_{n-1/2}(\chi)+\mathfrak{R}_n(\chi)\right) 
- |\gamma(t)-\gamma(t')|K_{\Gamma}(t,0,t',0)\mathfrak{R}_n(\chi)\right]
\end{equation}
and
\begin{equation*} 
S^n_\gamma(t, t') \;=\; \frac{1}{\sqrt{2\pi^3\gamma_1(t)\gamma_1(t')}} 
\mathfrak{Q}_{n-1/2}(\chi),\end{equation*}
where
$$\chi=1+\frac{|\gamma(t)-\gamma(t')|^2}{2\gamma_1(t)\gamma_1(t')},$$
$\mathfrak{Q}_{n-1/2}$ is a half-integer degree Legendre function of
the second kind,
$$\mathfrak{Q}_{n-1/2}(\chi) \;=\; \int_{-\pi}^\pi \frac{\cos(n\theta) \, d\theta}{\sqrt{8(\chi-\cos(\theta))}} = \int_{-\pi}^\pi \frac{e^{in\theta} \, d\theta}{\sqrt{8(\chi-\cos(\theta))}},$$
and
$$\mathfrak{R}_n(\chi) \;=\; \frac{2n-1}{\chi+1}\left(\chi\mathfrak{Q}_{n-1/2}(\chi)-\mathfrak{Q}_{n-3/2}(\chi) \right).$$
For $n < 0$, we have $K_{\gamma}^n(t,t') = K_{\gamma}^{-n}(t,t')$.

The following lemma will be proved in Section~\ref{app:Qasymp}. Note that $\mathfrak{Q}_{n-1/2}(1 + \delta^2) \geq 0$ for every $\delta > 0$ and $n \geq 0$.
\begin{lemma} \label{lem:Qasymp}
Let $n \in \mathbb{N}$ and $\delta > 0$.  If $\delta \geq \cone$, then
$$\mathfrak{Q}_{n-1/2}(1+\delta^2) \;\lesssim\; \frac{1}{n^2 \delta^3}, \quad  |\mathfrak{R}_n(1+\delta^2)| \;\lesssim\; \frac{1}{n^2\delta^3}.$$

If $\delta <2$, then
$$\mathfrak{Q}_{n-1/2}(1+\delta^2) \;\lesssim\; \frac{1}{(n\delta)^2}, \quad  |\mathfrak{R}_n(1+\delta^2)| \;\lesssim\; \frac{1}{(n\delta)^2}.$$

If $n\delta < 1/2$, then 
$$\mathfrak{Q}_{n-1/2}(1+\delta^2) \;\lesssim\; \log \frac{1}{n\delta}, \quad |\mathfrak{R}_n(1+\delta^2)| \;\lesssim\;  \log \frac{1}{n\delta}.$$
\end{lemma}

\begin{remark}
By successive integrations by part, it can actually be shown that
if $\delta \geq \cone$, then
$$\mathfrak{Q}_{n-1/2}(1+\delta^2) \;\lesssim\; \frac{1}{n^A \delta^3}, \quad  |\mathfrak{R}_n(1+\delta^2)| \;\lesssim\; \frac{1}{n^A\delta^3},$$
for any $A > 0$.
\end{remark}
\subsection{The essential spectrum from a conical singularity}
Suppose that $\Gamma$ is obtained by revolution of a $C^5$ curve $\gamma$, and that $\Gamma$ is $C^1$ except for one conical singularity that forms an angle $\alpha$ with the rotational axis.  An angle $\alpha\in(0,\frac{\pi}{2})$ corresponds to a conical singularity pointing outward, and an angle $\alpha\in(\frac{\pi}{2},\pi)$ corresponds to an inward-pointing conical singularity. 

The essential spectra of $K_\Gamma \colon \mathcal{E} \to \mathcal{E}$ and $K_\Gamma \colon L^2(\Gamma) \to L^2(\Gamma)$ were characterized in \cite{HelsingPerfekt2018a}. To state the result, let $\Pi^n_\alpha$, $n \in \mathbb{Z}$, denote the holomorphic function
\begin{equation} \label{eq:Pidef}
\Pi^n_\alpha(z) = \frac{P_{z-2}^{|n|}(\cos \alpha)\dot{P}_{z-2}^{|n|}(-\cos \alpha) - P_{z-2}^{|n|}(-\cos \alpha)\dot{P}_{z-2}^{|n|}(\cos \alpha)}{P_{z-2}^{|n|}(-\cos \alpha)\dot{P}_{z-2}^{|n|}(\cos \alpha)+P_{z-2}^{|n|}(\cos \alpha)\dot{P}_{z-2}^{|n|}(-\cos \alpha)}, \quad 0 < \mre z < 3,
\end{equation}
where $P^{|n|}_{z-2}(x)$ denotes an associated Legendre function of the first kind, and $\dot{P}_{z-2}^{|n|}(x)$ its derivative with respect to $x$.

\begin{theorem}[{\cite[Theorems~5.5 and 6.3]{HelsingPerfekt2018a}}] \label{thm:essspec}
Let $n \in \mathbb{Z}$. The essential spectrum of $K_{\gamma}^n \colon \mathcal{E}_n \to \mathcal{E}_n$ is a real interval,
$$\sigma_{\ess}(K_{\gamma}^n, \mathcal{E}_n) = \{\Pi^n_\alpha(3/2+ i\xi) \, : \, -\infty \leq \xi \leq \infty\}.$$
The essential spectrum of $K_{\gamma}^n \colon L^2\big([0,1],\, \gamma_1(t)|\gamma'(t)| \, dt\big) \to L^2\big([0,1],\, \gamma_1(t)|\gamma'(t)| \, dt\big)$ is a complex curve,
$$\sigma_{\ess}(K_{\gamma}^n, L^2\big([0,1],\, \gamma_1(t)|\gamma'(t)| \, dt\big)) = \{\Pi^n_\alpha(1+ i\xi) \, : \, -\infty \leq \xi \leq \infty\}.$$
\end{theorem}

	It follows from \cite[Theorem 7.5]{Bonnet-BenDhia2012} that $\sigma_{\ess}(K_\Gamma, \mathcal{E}) \subset [0,1]$ when $0 < \alpha < \frac{\pi}{2}$ and $\sigma_{\ess}(K_\Gamma, \mathcal{E}) \subset [-1,0]$ when $\frac{\pi}{2} < \alpha < \pi$; the link between the problem studied in \cite[Theorem 7.5]{Bonnet-BenDhia2012} and the spectral theory of $K_\Gamma$ can be found in \cite[Section~2]{BonnetierZhang2017}. That is, for each $n\in \mathbb Z$, $\sigma_{\ess}(K_{\gamma}^n, \mathcal{E}_n)$ is a positive interval containing $0$ when $0 < \alpha < \frac{\pi}{2}$, and a negative interval containing $0$ when  $\frac{\pi}{2} < \alpha < \pi$.

In Section~\ref{sec:est} we will clarify some of the calculations of \cite{HelsingPerfekt2018a} further. In accordance with Theorem~\ref{thm:essspec},
let $|\sigma_{n,\alpha}|$ denote the essential spectral radius of $K^n_\gamma \colon \mathcal{E}_n \to \mathcal{E}_n$,
	$$|\sigma_{n,\alpha}| = \sup \{|\Pi^n_\alpha(3/2+ i\xi)| \, : \, -\infty \leq \xi \leq \infty\},$$
and let $|\tilde{\sigma}_{n,\alpha}|$ denote the essential spectral radius of $K^n_\gamma$ acting on $L^2\big([0,1],\, \gamma_1(t)|\gamma'(t)| \, dt\big)$,
$$|\tilde{\sigma}_{n,\alpha}| = \sup \{|\Pi^n_\alpha(1+ i\xi)| \, : \, -\infty \leq \xi \leq \infty\}.$$
\begin{lemma} \label{lem:essest}
For  $0 < \alpha < \frac{\pi}{2}$ we have that $|\sigma_{0, \alpha}| = \Pi^0_\alpha(3/2)$ and 
\begin{equation} \label{eq:Ktotspec}
\sigma_{\ess}(K_\Gamma, \mathcal{E}) = [0, |\sigma_{0, \alpha}|].
\end{equation}
 For $\frac{\pi}{2} < \alpha < \pi$ we instead have that $|\sigma_{0, \alpha}| = -\Pi^0_\alpha(3/2)$ and $\sigma_{\ess}(K_\Gamma, \mathcal{E}) = [-|\sigma_{0, \alpha}|, 0]$.

For every $n \in \mathbb{Z}$ and $\alpha \neq \pi/2$, we have that
$$|\sigma_{n,\alpha}| \leq |\tilde{\sigma}_{n,\alpha}|.$$
Furthermore, for every $c_0\in (0,\frac{\pi}{2})$ there exists a constant $C_0$ such that if $|\alpha -\frac{\pi}{2}| < c_0$, then
	\begin{equation}\label{eq:bdess}
|\tilde{\sigma}_{n,\alpha}| \leq C_0 \frac{|\alpha - \pi/2|}{1+|n|}, \qquad n \in \mathbb{Z}.
	\end{equation}
\end{lemma}
	
	Suppose that $\Gamma$ is perturbed smoothly so that the conical singularity is retained but the rotational symmetry is destroyed, by applying a smooth diffeomorphism $\psi$ of $\mathbb{R}^3$ which is conformal at the vertex of $\Gamma$, $\Gamma' = \psi(\Gamma)$. Then the techniques of \cite{Med97} can be adapted to the $L^2$-setting to view $K_{\Gamma'} \colon L^2(\Gamma') \to L^2(\Gamma')$ as a compact perturbation of an operator similar to $K_{\Gamma} \colon L^2(\Gamma) \to L^2(\Gamma)$. Appealing to extrapolation of compactness as in \cite[Theorem~5.22]{LCP21}, the same statement is true when $L^2(\Gamma')$ and $L^2(\Gamma)$ are replaced by $\mathcal{E}(\Gamma')$ and $\mathcal{E}(\Gamma)$, respectively.  Equation~\eqref{eq:Ktotspec} therefore remains valid for the perturbed surface $\Gamma'$. However, we expect any embedded eigenvalues to disappear in general, since symmetry is instrumental for their construction.  

\subsection{The spectrum for a sphere} \label{sec:spherespec}
Let $\Gamma_0$ be the unit sphere, obtained by revolution of $\gamma_0$. The spherical harmonics are in spherical coordinates defined as
\begin{equation*}
Y^m_{\ell} (\beta, \theta) = \sqrt{\frac{2{\ell}+1}{4\pi}\frac{({\ell}-m)!}{({\ell}+m)!}\,}\,P^m_{\ell}(\cos\beta)e^{im\theta}, \quad {\ell} \geq 0, \quad -{\ell}\leq m \leq {\ell}, \; 0<\beta<\pi, \; 0<\theta<2\pi.
\end{equation*} 
They form an orthonormal Hilbert-space basis for $L^2(\Gamma_0)$,
$$
 \int_{\Gamma_0} Y^m_{\ell} (\beta, \theta) \overline{Y^{m'}_{{\ell}'} (\beta, \theta)} \, d\sigma= \delta_{m,m'}\delta_{{\ell},{\ell}'}.
$$
Note that $Y^m_{\ell} $ belongs to the $m$th Fourier space.

Since $\Gamma_0$ is smooth, the spectrum of $K_{\gamma_0}^n$ is the same whether considered as an operator on $L^2([0,1], \gamma_{0,1}(t)|\gamma_0'(t)|dt)$ or on $\mathcal E_n$ \cite{KhavinsonPutinarShapiro2007},
$$
\sigma(K_{\gamma_0}^n) = \left\{ \frac{1}{2{\ell}+1} \;:\; {\ell}\geq |n| \right\} \cup \{0\}.
$$
The eigenvalue $\frac{1}{2{\ell}+1}$ of $K_{\gamma_0}^n$ has eigenfunction $Y^n_{\ell}(\cdot,0)$. The point $0$ of the spectrum is not an eigenvalue. 

In our analysis, we will additionally rely on a very special property for the sphere, namely that 
$$S_{\Gamma_0} = 2K_{\Gamma_0},$$
and therefore that $S_{\gamma_0}^n = 2 K_{\gamma_0}^n$ for all $n \in \mathbb{Z}$.


\section{The NP operator on perturbations of the unit sphere} \label{sec:main}

This section introduces small conical perturbations of the sphere and investigates the corresponding perturbation of eigenvalues. It concludes with a proof of the main Theorem~\ref{thm:infemb}, up to a number of technical estimates deferred to the next section.

\subsection{Perturbation and parametrization}\label{sec:pert}
This subsection sets the notation for our perturbations of the sphere.  The type of perturbation that we will consider amounts to the smooth addition of a small rotationally symmetric conical singularity with a shallow angle.  The notation established here is maintained throughout.

Let $\Gamma_0 \subset \mathbb{R}^3$ be the unit sphere. We fix a parametrization $\gamma_0(t) = (\gamma_{0,1}(t), \gamma_{0,2}(t))$, $0 \leq t \leq 1$, of its generating curve that satisfies the following requirements:

\begin{enumerate}
\item[A-1.]
$\gamma_{0,2}(0) = -1$ and $\gamma_{0,2}(1) = 1$;
\item[A-2.]
$\gamma_{0,1}$ and $\gamma_{0,2}$ are smooth;
\item[A-3.]
Near $t=0$ the parametrization is of the form of a graph: $\gamma_0(t)=(t, \gamma_{0,2}(t))$, $t\in[0,\frac{1}{5}]$;
\item[A-4.]
$|\gamma_0'(t)| \in[\frac{1}{4},4]$ for  $t\in[0,1]$. 
\end{enumerate}
\begin{remark}\label{rmk:beta}
With this parametrization, the Euler angle representing the tilt $\beta$ is a function of $t$. We denote this function by $\beta(t)$.  By A-3, $\gamma_{0,2}(t) = -\sqrt{1-t^2}$ for $t\in[0,\frac{1}{5}]$.
\end{remark}

\begin{figure}
  \centerline{\scalebox{0.45}{\includegraphics{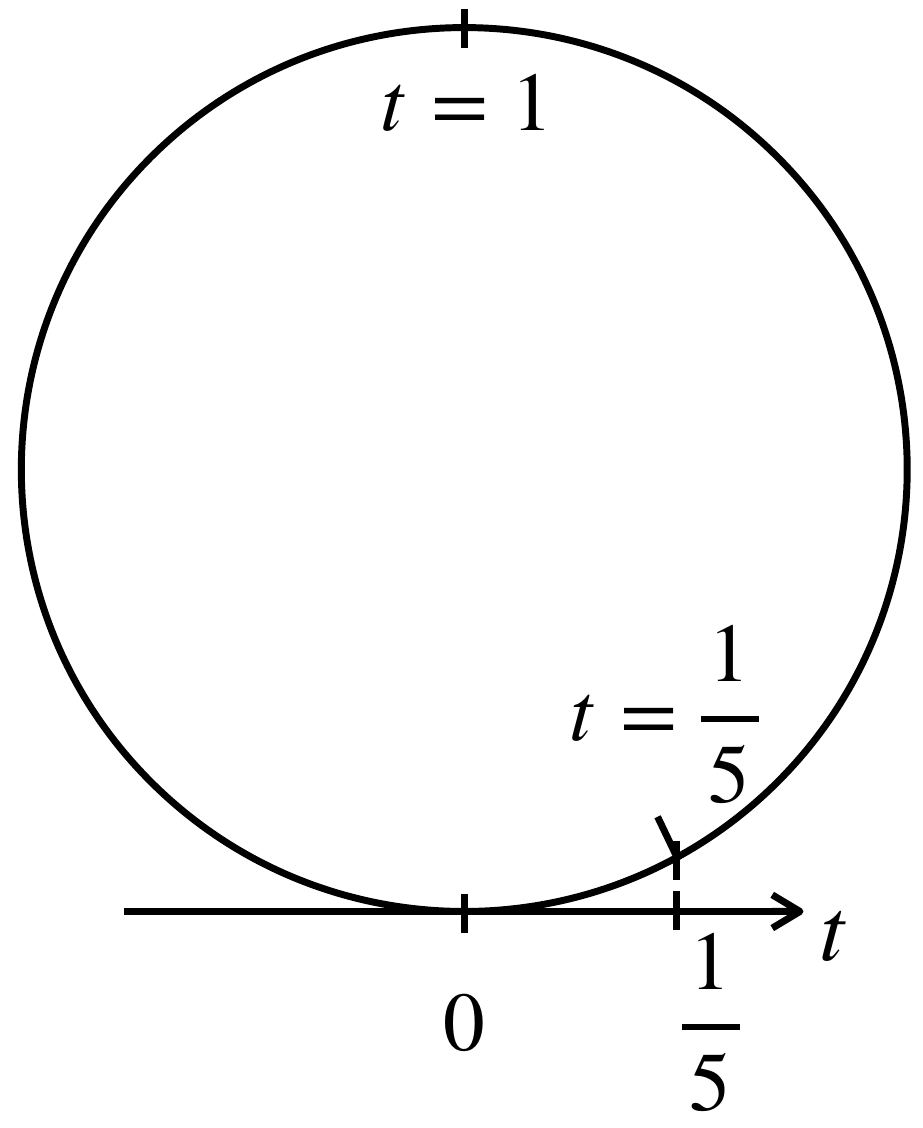}}}
  \caption{The profile of the unit sphere $\Gamma_0$.}
  \label{fig:gamma0}
\end{figure}

Next we perturb $\gamma_0$ in a small neighborhood of $t = 0$. For a given angle $\alpha$, we let $\varepsilon = \frac{1}{8}|\alpha-\frac{\pi}{2}|$. Let $c_0 > 0$ be so small that whenever $|\alpha-\frac{\pi}{2}| < c_0$, it holds that
\begin{equation}\label{eq:pertbd}
 6\varepsilon<\frac{1}{5},\quad \gamma_{0,2}'(6\varepsilon) = \frac{6 \epsilon}{\sqrt{1-(6\varepsilon)^2}}<|\!\cot\alpha| < \sqrt{8}, \text{ and } |\!\cot\alpha| \leq \frac{9}{8} \left|\alpha-\frac{\pi}{2}\right|.
\end{equation}

\begin{definition} \label{def:pert}
Let $\Gamma_0$ be the unit sphere with a parametrization satisfying Assumptions A-1 through A-4. Let  $\alpha$ and $\varepsilon$ satisfy $|\alpha -\frac{\pi}{2}|<c_0$ and $\varepsilon = \frac{1}{8}|\alpha -\frac{\pi}{2}|$. We say that $\Gamma$ is an $(\alpha,\varepsilon)$-perturbation of the unit sphere $\Gamma_0$ if its generating curve $\gamma$ satisfies
\begin{enumerate}
\item[B-1.]
$\gamma_1, \gamma_2 \in C^{\infty}([0,1])$;
\item[B-2.]
$\gamma(t) = \gamma_0(t)$ for $t\in[\varepsilon,1]$;
\item[B-3.]
$\gamma(t)=(t, \gamma_2(t))$ for $t\in[0,\varepsilon]$;
\item[B-4.]
$\gamma_2'(t) =  \cot(\alpha)$ for $t\in[0,\frac{\varepsilon}{2}]$;
\item[B-5.]
$|\gamma_2'(t)|\leq|\cot\alpha|$ for $t\in[\frac{\varepsilon}{2},\varepsilon]$;
\item[B-6.]
$|\gamma_2''(t)|\;\leq\; 40$ for $t\in[0,\varepsilon]$;
\end{enumerate}
\end{definition}

The definition implies that the curvature of $\gamma$ is uniformly bounded by $40$ and that $|\gamma'(t)| \in [\frac{1}{4},4]$ for $t\in[0,1]$. The surface $\Gamma$ obtained by revolving $\gamma$ has a conical singularity pointing outward when $0 < \alpha<\pi/2$, and a conical singularity pointing inward when $\pi/2 < \alpha<\pi$.

Note that $(\alpha,\varepsilon)$-perturbations of $\Gamma_0$ exist for every $\alpha$ satisfying $|\alpha-\frac{\pi}{2}|<c_0$. In particular, it is possible to satisfy B-6 because $\gamma_2'$ needs to change by at most $2|\cot\alpha| \leq \frac{9}{4}|\alpha-\frac{\pi}{2}|$ within the interval $[\varepsilon/2, \varepsilon]$ of length $\frac{\varepsilon}{2} = \frac{1}{16}|\alpha-\frac{\pi}{2}|$.

\begin{remark}
  The parameters $\varepsilon$ and $\alpha$ are tied through $\varepsilon = \frac{1}{8}|\alpha -\frac{\pi}{2}|$ in order to control the curvature of the perturbation.  Nevertheless, we distinguish the two parameters in our results in order to separate their roles in the analysis.
\end{remark}

\begin{figure}
  \centerline{\scalebox{0.45}{\includegraphics{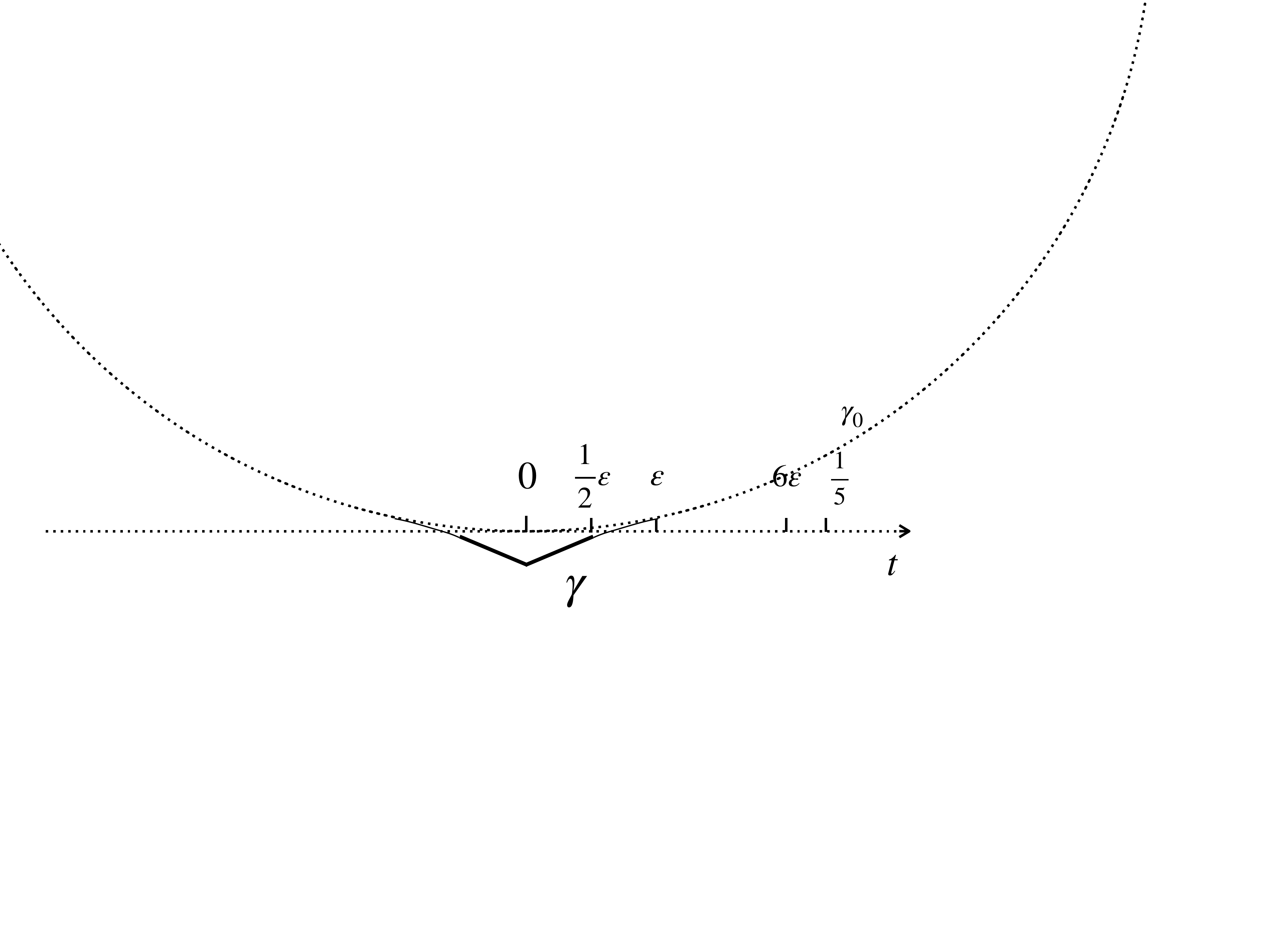}}}
  \caption{The profile of an $(\alpha,\varepsilon)$-perturbation $\Gamma$ of the unit sphere $\Gamma_0$.}
  \label{fig:gamma}
\end{figure}

\subsection{Approximate eigenpairs}\label{sec:approx}
In this section, for $n \geq 1$, let
\begin{equation} \label{eq:eigpair}
\lambda^n=\frac{1}{2n+1}, \quad f_n (t) = Y^n_n(\beta(t),0),
\end{equation}
where $\beta(t)$ is the function defined in Remark \ref{rmk:beta}.  This defines an exact eigenpair for $K_{\gamma_0}^n$.  The goal of this section is to show that, for infinitely many $n$, this is an approximate eigenpair for $K_{\gamma}^n$ for $(\alpha,\epsilon)$-perturbations with $|\alpha-\frac{\pi}{2}|$ sufficiently small. 

For convenience, we write
$$
d\mu_0(t):= \gamma_{0,1}(t)|\gamma'_0(t)| \, dt, \quad d\mu(t) = \gamma_1(t)|\gamma'(t)| \, dt.
$$
When $\Gamma$ is an $(\alpha,\varepsilon)$-perturbation of $\Gamma_0$, the norms of $L^2([0,1], d\mu)$ and $L^2([0,1], d\mu_0)$ are equivalent,
$$
\frac{1}{4^2}\|f\|_{L^2([0,1],d\mu_0)} \;\leq\; \|f\|_{L^2([0,1],d\mu)} \;\leq\; 4^2\|f\|_{L^2([0,1],d\mu_0)},
$$
Any operator $T$ on $L^2([0,1], d\mu)$, for example $T = K_\gamma^n$, can therefore be understood as an operator on $L^2([0,1], d\mu_0)$ of equivalent norm,
\begin{equation*}
 \frac{1}{4^4}  \|T \|_{\mu} 
 \;\leq\;  \|T \|_{\mu_0} 
 \;\leq\; 4^4  \|T \|_{\mu},
\end{equation*}
where $\|\!\cdot\!\|_{\mu} = \|\!\cdot\!\|_{L^2([0,1], d\mu)\to L^2([0,1], d\mu)}$ and $\|\!\cdot\!\|_{\mu_0} = \|\!\cdot\!\|_{L^2([0,1], d\mu)\to L^2([0,1], d\mu_0)}$.

In particular, we consider $K_{\gamma}^n -K_{\gamma_0}^n$ and $S_{\gamma}^n -S_{\gamma_0}^n$ as operators on $L^2([0,1],d\mu_0)$. Explicitly, 
\begin{equation*}
	(K_{\gamma}^n -K_{\gamma_0}^n )f (t) =  \int_0^1 \mathfrak{K}^n (t,t') f(t') \, d\mu_0 (t'), \quad
	(S_{\gamma}^n -S_{\gamma_0}^n )f(t) =  \int_0^1 \mathfrak{S}^n (t,t') f(t')\, d\mu_0 (t'),
\end{equation*}
where the kernels are given by
\begin{equation}\label{eq:kernelKS}
\mathfrak{K}^n (t,t') =  K_{\gamma}^n(t,t') \frac{\gamma_1(t')|\gamma'(t')|} {\gamma_{0,1}(t')|\gamma_0'(t')|}  - K_{\gamma_0}^n(t,t'), \quad 
\mathfrak{S}^n (t,t') =  S_{\gamma}^n(t,t')  \frac{\gamma_1(t')|\gamma'(t')|} {\gamma_{0,1}(t')|\gamma_0'(t')|}   - S_{\gamma_0}^n(t,t') .
\end{equation}

Similarly, we consider $K_{\Gamma} - K_{\Gamma_0}$ and $S_{\Gamma}  -S_{\Gamma_0}$ on $L^2(\Gamma_0) = L^2([0,1] \times [0, 2\pi), d\mu_0(t) \, d\theta)$,
\begin{equation*}
	\begin{aligned}
		(K_{\Gamma}  - K_{\Gamma_0} )f (\br_0(t,\theta))=  \int_{\Gamma_0} \mathfrak{K} (\br_0(t,\theta),\br_0(t',\theta')) f(\br_0(t',\theta')) \, d\mu_0(t) \, d\theta, \\
		(S_{\Gamma}  -S_{\Gamma_0} )f (\br_0(t,\theta)) =  \int_{\Gamma_0} \mathfrak{S}(\br_0(t,\theta),\br_0(t',\theta')) f(\br_0(t',\theta')) \, d\mu_0(t) \, d\theta.
	\end{aligned}
\end{equation*}
Their kernels are given by
\begin{equation*}
	\begin{aligned}
		\mathfrak{K} (\br_0(t,\theta),\br_0(t',\theta')) =  K_{\Gamma}(\br(t,\theta),\br(t',\theta'))  \frac{\gamma_1(t')|\gamma'(t')|} {\gamma_{0,1}(t')|\gamma_0'(t')|}  - K_{\Gamma_0}(\br_0(t,\theta),\br_0(t',\theta')) , \\
		\mathfrak{S} (\br_0(t,\theta),\br_0(t',\theta')) =  S_{\Gamma}(\br(t,\theta),\br(t',\theta'))  \frac{\gamma_1(t')|\gamma'(t')|} {\gamma_{0,1}(t')|\gamma_0'(t')|}  - S_{\Gamma_0}(\br_0(t,\theta),\br_0(t',\theta')) .
	\end{aligned}
\end{equation*}
Here $\br(t,\theta)$ and $\br_0(t,\theta)$ refer to the parametrizations of $\Gamma$ and $\Gamma_0$ induced by $\gamma$ and $\gamma_0$, respectively.

Let $\chi_1 + \chi_2 = \tilde\chi_1 + \tilde\chi_2  = 1$ be two partitions of unity of $[0,1]$ such that 
$\mathrm{supp}\,\chi_1 \subset [0,4\varepsilon]$, $\mathrm{supp}\,\chi_2\subset [3\varepsilon,1]$, $\mathrm{supp}\,\tilde\chi_1 \subset [0,2\varepsilon]$, and $\mathrm{supp}\,\tilde\chi_2\subset [\varepsilon,1]$,
as in Figure~\ref{fig:partition}. 
\begin{figure}
	\centerline{\scalebox{0.45}{\includegraphics{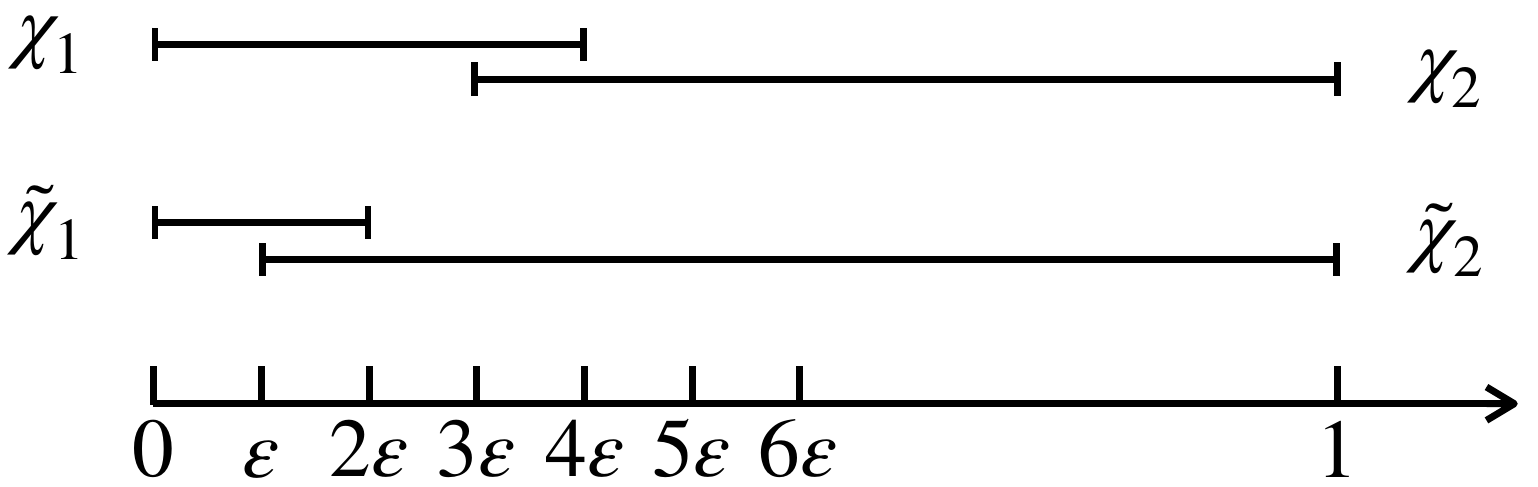}}}
	\caption{Two pairs of partition of unity}
	\label{fig:partition}
\end{figure}
Since $ \chi_2 \mathfrak{K}^n \tilde\chi_2 = 0$ and $ \chi_2 \mathfrak{S}^n \tilde\chi_2 = 0$, we have
\begin{equation}\label{eq:split}
\begin{aligned}
\mathfrak{K}^n &= \chi_1 \mathfrak{K}^n + \chi_2 \mathfrak{K}^n \tilde\chi_1,\\
\mathfrak{S}^n &= \chi_1 \mathfrak{S}^n + \chi_2 \mathfrak{S}^n \tilde\chi_1. 
\end{aligned}
\end{equation}

We now state two technical lemmas, deferring their proofs to Section~\ref{sec:est}. Given a kernel $\mathfrak{T}$, we also denote the operator it induces by $\mathfrak{T}$.
\begin{lemma} \label{lem:localest}
	There exists a constant $C > 0$ such that 
	for every $(\alpha,\varepsilon)$-perturbation $\Gamma$ of $\Gamma_0$, the operators $\chi_1 \mathfrak{K}^n$ and $\chi_1 \mathfrak{S}^n$ satisfy
	\begin{equation*}
	\| \chi_1 \mathfrak{K}^n \|_{\mu_0 } \;\leq\; C \frac{|\alpha - \pi/2| + \varepsilon}{n}, \quad 
	\|\chi_1 \mathfrak{S}^n \|_{\mu_0} \;\leq\; C \frac{\varepsilon}{n}.
	\end{equation*}
\end{lemma}

\begin{lemma} \label{lem:Schatten}
	For every $(\alpha,\varepsilon)$-perturbation $\Gamma$ of $\Gamma_0$, the operators $\chi_2 \mathfrak{K} \tilde{\chi}_1$ and $\chi_2 \mathfrak{S} \tilde{\chi}_1$ are in the Schatten class $S_p(L^2(\Gamma_0))$ for some $p < 1$.
\end{lemma}

As a consequence of Lemmas~\ref{lem:localest} and \ref{lem:Schatten}, we have the following.

\begin{lemma}\label{lemma:keyKS0}
There exists a constant $C>0$ satisfying the following.
For every $(\alpha,\varepsilon)$-perturbation $\Gamma$ of $\Gamma_0$ there exists an infinite set $Z \subset \mathbb{N}$ such that 
 \begin{equation*}
 \begin{aligned}
& \| K_{\gamma}^n -K_{\gamma_0}^n  \|_{\mu_0} \;\leq\;  C\frac{|\alpha-\pi/2| + \varepsilon }{n}, \\
& \| S_{\gamma}^n -S_{\gamma_0}^n \|_{\mu_0} \;\leq\; C \frac{ \varepsilon }{n},
\end{aligned}
\end{equation*}
for every $n \in Z$.
\end{lemma}

\begin{proof}
Every norm $\| \chi_2 \mathfrak{K}^n \tilde\chi_1\|_{\mu_0}$ appears as a singular value of $\chi_2 \mathfrak{K} \tilde\chi_1$, an operator on $L^2(\Gamma_0) = L^2([0,1] \times [0, 2\pi), d\mu_0 \, d\theta)$. The same statement is also true with $\mathfrak{S}$ in place of $\mathfrak{K}$. From Lemma~\ref{lem:Schatten} we thus know that there is a $p < 1$ such that
$$
\sum_n \| \chi_2 \mathfrak{K}^n \tilde\chi_1\|_{\mu_0}^p <\infty, \quad  \sum_n \| \chi_2 \mathfrak{S}^n \tilde\chi_1\|_{\mu_0}^p <\infty.
$$
In particular, there are infinitely many $n$ for which
 $$
 \| \chi_2 \mathfrak{K}^n \tilde\chi_1\|_{\mu_0} <\frac{1}{n^{\frac{1}{p}}}, \quad \| \chi_2 \mathfrak{S}^n \tilde\chi_1\|_{\mu_0} <\frac{1}{n^{\frac{1}{p}}}.
$$
The lemma now follows from \eqref{eq:split} and Lemma~\ref{lem:localest}.
\end{proof}

We can now prove the main result of this section. In the statement, $\lambda^n$ and $f_n$ are as in \eqref{eq:eigpair}, so that $K^n_{\gamma_0} f_n = \lambda^n f_n$.
\begin{lemma}\label{lemma:approxevalue}
There exist $c \in(0,\frac{\pi}{2})$ and $C>0$ such that, 
for every $(\alpha,\varepsilon)$-perturbation $\Gamma$ of~$\Gamma_0$ with $|\alpha -\frac{\pi}{2}|<c$, there exist infinitely many $n$ for which
\begin{equation*}
\|(K_{\gamma}^n - \lambda^n) f_n\|_{S^n_\gamma} \;\leq\;   C\frac{|\alpha-\pi/2|+\varepsilon}{n}  \| f_n\|_{S^n_\gamma}.
\end{equation*}
\end{lemma}

\begin{proof}
Let $\langle\cdot,\cdot \rangle_\mu$ denote the inner product of $L^2([0,1],d\mu)$. By Lemma~\ref{lemma:keyKS0}, there is a set $Z = Z(\Gamma) \subset \mathbb{N}$ such that
\begin{equation} \label{eq:approxeigchain}
\begin{aligned}
&\|(K_{\gamma}^n - \lambda^n) f_n\|_{S^n_\gamma}^2  = \langle S_{\gamma}^n (K_{\gamma}^n -K_{\gamma_0}^n) f_n,(K_{\gamma}^n - K_{\gamma_0}^n) f_n\rangle_\mu \leq \\
& \|S_{\gamma_0}^n (K_{\gamma}^n -K_{\gamma_0}^n) f_n\|_\mu \|(K_{\gamma}^n - K_{\gamma_0}^n) f_n\|_\mu +  \|( S_{\gamma}^n -S_{\gamma_0}^n) (K_{\gamma}^n -K_{\gamma_0}^n) f_n\|_\mu \|(K_{\gamma}^n - K_{\gamma_0}^n) f_n\|_\mu \lesssim \\
& \frac{1}{n} \| (K_{\gamma}^n -K_{\gamma_0}^n) f_n \|^2_\mu \lesssim \frac{(|\alpha-\pi/2| + \varepsilon)^2}{n^3} \|f_n \|^2_\mu
\end{aligned}
\end{equation}
whenever $n \in Z$. In addition to Lemma~\ref{lemma:keyKS0}, we have here made use the fact that
$$
\|S_{\gamma_0}^n\|_{\mu} \;\leq\; 4^4 \| S_{\gamma_0}^n \|_{\mu_0} = 4^4 \frac{2}{2n+1},
$$
which follows from the fact that $S_{\gamma_0}^n \colon L^2([0,1],d\mu_0) \to L^2([0,1],d\mu_0)$ is a non-negative operator with $2/(2n+1)$ as its largest eigenvalue (Section~\ref{sec:spherespec}).

On the other hand, for $n \in Z$,
\begin{align*}
\frac{2}{2n+1}  \|f_n \|^2_\mu &=   \langle S_{\gamma_0}^n f_n, f_n\rangle_{\mu} \leq \langle S_{\gamma}^n f_n ,f_n\rangle_{\mu} +  
\|(S_{\gamma_0}^n  - S_{\gamma}^n) f_n\|_{\mu} \|f_n\|_\mu \\
&\lesssim \langle S_{\gamma}^n f_n ,f_n\rangle_{\mu} +   \frac{\varepsilon}{n} \|f_n\|^2.
\end{align*}
Here we have recalled that $ \left< S_{\gamma}^n f_n ,f_n\right>  = \|f_n\|^2_{S_\gamma^n}\geq 0$. If $|\alpha - \frac{\pi}{2}|$, and thus $\varepsilon$, is sufficiently small, we conclude that
\begin{equation*}
\|f_n \|^2_{\mu} \lesssim n \|f_n\|^2_{S_\gamma^n}, \quad n \in Z.
\end{equation*}
Inserting this estimate into \eqref{eq:approxeigchain} yields the desired conclusion.
\end{proof}

\subsection{Proof of Theorem~\ref{thm:infemb}}
We can now give the proof of Theorem~\ref{thm:infemb}. Recall that the proofs of Lemmas~\ref{lem:Qasymp}, \ref{lem:essest}, \ref{lem:localest}, and \ref{lem:Schatten} have been deferred to the next section.
\begin{proof}
	Let $c$ and $C$ be as in Lemma~\ref{lemma:approxevalue}, and let $\Gamma$ be an $(\alpha, \varepsilon)$-perturbation of $\Gamma_0$ with $0 < \frac{\pi}{2} - \alpha < c$. Then Lemma~\ref{lemma:approxevalue} shows that there is an infinite set $Z = Z(\Gamma) \subset \mathbb{N}$ for which 
	\begin{equation*}
	\mathrm{dist}(\sigma(K_{\gamma}^n,  \mathcal{E}_n ),  1/(2n+1)) \leq C \frac{\pi/2 - \alpha+\varepsilon}{n}, \quad n \in Z.
	\end{equation*}
By Lemma~\ref{lem:essest}, we also know that the essential spectral radius of $K_{\gamma}^n \colon \mathcal{E}_n \to \mathcal{E}_n$ satisfies that
$$|\sigma_{n,\alpha}| \leq C_0 \frac{\pi/2 - \alpha}{n} $$
for some constant $C_0$.
Hence, if $\frac{\pi}{2} - \alpha$, and therefore $\varepsilon$, is chosen sufficiently small, $K^n_\gamma \colon \mathcal{E}_n \to \mathcal{E}_n$ must for every $n \in Z$ have an eigenvalue $z_n$ satisfying
$$|\sigma_{n,\alpha}| < \frac{1}{2n+1} - C \frac{\pi/2 - \alpha+\varepsilon}{n} \leq z_n \leq \frac{1}{2n+1} + C \frac{\pi/2 - \alpha+\varepsilon}{n}.$$
Every $z_n$ is an eigenvalue of $K_\Gamma \colon \mathcal{E} \to \mathcal{E}$, infinitely many of which are embedded in
\begin{equation*}
\sigma_{\ess}(K_\Gamma, \mathcal{E}) = [0, |\sigma_{0, \alpha}|]. \qedhere
\end{equation*}
\end{proof}
\begin{remark}
	By instead considering $(\alpha, \varepsilon)$-perturbations with $\frac{\pi}{2} < \alpha < \pi$, the same argument yields examples $\Gamma$ where $K_\Gamma \colon \mathcal{E} \to \mathcal{E}$ exhibits infinitely many eigenvalues outside its essential spectrum, which is negative in this case. Such an example has previously been observed numerically in \cite[Section~7.3.3]{HelsingPerfekt2018a}.
\end{remark}
\section{Kernel and operator estimates} \label{sec:est}

This section contains the proofs of Lemmas \ref{lem:Qasymp}, \ref{lem:essest}, \ref{lem:localest}, and \ref{lem:Schatten}.

\subsection{Proof of Lemma~\ref{lem:Qasymp}} \label{app:Qasymp}
\begin{proof}
We first consider
\begin{align*}
 \mathfrak{Q}_{n-1/2}(1+\delta^2) 
 = \int_{-\pi}^\pi \frac{e^{in\theta} \, d\theta}{\sqrt{\delta^2 + (1-\cos(\theta))}}
 = \int_{-\pi}^\pi \frac{\cos(n\theta)\, d\theta}{\sqrt{\delta^2 + (1-\cos(\theta))}}\,.
\end{align*}
Integration by parts once, twice, and three times gives
\begin{equation}\label{eq:IBPQ}
\begin{aligned}
 \mathfrak{Q}_{n-1/2}(1+\delta^2) 
 &= \frac{1}{2n} \int_{-\pi}^{\pi} \frac{\sin(ny)\sin(y)\, dy}{(\delta^2 + (1-\cos(y)))^{3/2}}\\
 &= \frac{1}{2n^2} \int_{-\pi}^{\pi} \frac{\cos(ny)\cos(y)\, dy}{(\delta^2 + (1-\cos(y)))^{3/2}} - \frac{3}{4n^2} \int_{-\pi}^{\pi} \frac{\cos(ny)\sin^2(y)\, dy}{(\delta^2 + (1-\cos(y)))^{5/2}}\\
 &= \frac{1}{2n^3} \int_{-\pi}^{\pi} \frac{\sin(ny)\sin(y)\, dy}{(\delta^2 + (1-\cos(y)))^{3/2}} 
 +  \frac{3}{4n^3} \int_{-\pi}^{\pi} \frac{\sin(ny)\cos(y)\sin(y)\, dy}{(\delta^2 + (1-\cos(y)))^{5/2}}\,+ \\
 &\quad +\,  \frac{3}{2n^3} \int_{-\pi}^{\pi} \frac{\sin(ny)\sin(y)\cos(y)\, dy}{(\delta^2 + (1-\cos(y)))^{5/2}} 
 -  \frac{15}{8n^3} \int_{-\pi}^{\pi} \frac{\sin(ny)(\sin(y))^3\, dy}{(\delta^2 + (1-\cos(y)))^{7/2}} \,.
\end{aligned}
\end{equation}
From the third line of \eqref{eq:IBPQ} we see that
$$
\mathfrak{Q}_{n-1/2}(1+\delta^2) \;\lesssim\; \frac{1}{n^3 \delta^3} +  \frac{1}{n^3 \delta^5} +  \frac{1}{n^3 \delta^7}\,.
$$
Thus, for $\delta \geq \cone$, we obtain that
\begin{equation}\label{eq:q1}
\mathfrak{Q}_{n-1/2}(1+\delta^2) \;\lesssim\; \frac{1}{n^3 \delta^3} \leq \frac{1}{n^2 \delta^3}.
\end{equation}
Rescaling the integrals in the third line of \eqref{eq:IBPQ}, $y\rightarrow y/n$, yields
\begin{equation*}
\begin{aligned}
 \mathfrak{Q}_{n-1/2}(1+\delta^2) =
 &\frac{1}{2n^3} \int_{-\pi n}^{\pi n} \frac{\sin(y)\sin(y/n)\, dy/n}{(\delta^2 + (1-\cos(y/n)))^{3/2}} 
 +  \frac{3}{4n^3} \int_{-\pi n}^{\pi n} \frac{\sin(y)\cos(y/n)\sin(y/n)\, dy/n}{(\delta^2 + (1-\cos(y/n)))^{5/2}} \\
 &+  \frac{3}{2n^3} \int_{-\pi n}^{\pi n} \frac{\sin(y)\sin(y/n)\cos(y)\, dy/n}{(\delta^2 + (1-\cos(y/n)))^{5/2}} 
 -  \frac{15}{8n^3} \int_{-\pi n}^{\pi n} \frac{\sin(y)(\sin(y/n))^3\, dy/n}{(\delta^2 + (1-\cos(y/n)))^{7/2}} \,.
\end{aligned}
\end{equation*}
Since $1-\cos(y/n) \gtrsim (y/n)^2$, for all $n$ and $y\in(-\pi n,\pi n)$, we obtain that for $\delta>0$,
\begin{equation}\label{eq:q22}
\begin{aligned}
 \mathfrak{Q}_{n-1/2}(1+\delta^2) &\;\lesssim\;
 \frac{1}{n^3} \int_{-\infty}^{\infty} \frac{ dy/n}{(\delta^2 + (y/n)^2)^{3/2}} 
 +  \frac{1}{n^3} \int_{-\infty}^{\infty} \frac{|y|/n\, dy/n}{(\delta^2 +(y/n)^2))^{5/2}}\;+ \\
 &\qquad +  \frac{1}{n^3} \int_{-\infty}^{\infty} \frac{|y|/n \, dy/n}{(\delta^2 + (y/n)^2))^{5/2}} 
 +  \frac{1}{n^3} \int_{-\infty}^{\infty} \frac{(|y|/n)^3\, dy/n}{(\delta^2 +(y/n)^2))^{7/2}} \\
 &\;\lesssim\; \frac{1}{n^3\delta^2} +  \frac{1}{n^3\delta^3} \,.
\end{aligned}
\end{equation}

Scaling instead the second line of \eqref{eq:IBPQ}, we obtain
\begin{align*}
 \mathfrak{Q}_{n-1/2}(1+\delta^2) 
\;=\;  \frac{1}{2n^3} \int_{-\pi n}^{\pi n} \frac{\cos(y)\cos(y/n)\, dy}{(\delta^2 + (1-\cos(y/n)))^{3/2}} - \frac{3}{4n^3} \int_{-\pi n}^{\pi n} \frac{\cos(y)\sin^2(y/n)\, dy}{(\delta^2 + (1-\cos(y/n)))^{5/2}}\,.
\end{align*}
Thus for $ \delta>0$,
\begin{equation*}
 \mathfrak{Q}_{n-1/2}(1+\delta^2)  \;\lesssim\; \frac{1}{n^3} \int_{-\infty}^{\infty} \frac{dy}{(\delta^2 + (y/n)^2)^{3/2}}  +  \frac{1}{n^3} \int_{-\infty}^{\infty} \frac{(y/n)^2 \, dy}{(\delta^2 + (y/n)^2)^{5/2}}
\;\approx\; \frac{1}{(n\delta)^2}\,.
\end{equation*}

Scaling the first line of \eqref{eq:IBPQ}, we obtain
\begin{align*}
 \mathfrak{Q}_{n-1/2}(1+\delta^2) 
 = \frac{1}{2n^2} \int_{-\pi n}^{\pi n} \frac{\sin(y)\sin(y/n)\, dy}{(\delta^2 + (1-\cos(y/n)))^{3/2}}.
\end{align*}
Thus for $n\delta < 1/2$,
\begin{equation}\label{eq:q3}
\begin{aligned}
 \mathfrak{Q}_{n-1/2}(1+\delta^2)  &\lesssim\;
 \frac{1}{n^2} \bigg(\int_{-1}^{1} + \int_{|y| > 1} \bigg) \left|  \frac{\sin(y)\sin(y/n)}{(\delta^2 + (y/n)^2)^{3/2}} \right|\, dy\\
&\leq\; \frac{1}{n^2} \int_{-1}^{1}\frac{ y^2/n \, dy}{(\delta^2 + (y/n)^2)^{3/2}}  \;+\; \frac{1}{n^2}  \int_{|y| > 1} \frac{ |y|/n \, dy}{(\delta^2 + (y/n)^2)^{3/2}} \\
&=\; \int_{-1}^{1} \frac{y^2 \, dy}{((n\delta)^2 + y^2)^{3/2}} + \int_{|y| > 1} \frac{|y| \, dy}{((n\delta)^2 + y^2)^{3/2}} \\
&\approx\; \log \frac{1}{n\delta} + \frac{1}{\sqrt{n^2\delta^2 + 1}} 
 \;\lesssim\; \log \frac{1}{n\delta}\,.
\end{aligned}
\end{equation}
We have now proved the desired estimates for $\mathfrak{Q}_{n-1/2}(1+\delta^2)$.

Next we consider
$$\mathfrak{R}_n(\chi) \;=\; \frac{2n-1}{2+\delta^2}\left(\delta^2\mathfrak{Q}_{n-1/2}(\chi)+\mathfrak{Q}_{n-1/2}(\chi)-\mathfrak{Q}_{n-3/2}(\chi) \right).$$
Since $2n-1 \;\approx\; n$ and $2+\delta^2\geq 1$, it suffices to show that the three estimates in the lemma are satisfied~by
$$\frac{n\delta^2}{2+\delta^2}\,\mathfrak{Q}_{n-1/2}(\chi) \quad \text{and} \quad n \left(\mathfrak{Q}_{n-1/2}(\chi)-\mathfrak{Q}_{n-3/2}(\chi) \right).$$

Using \eqref{eq:q1}, we obtain that for $\delta> \cone$,
\begin{equation*}
 \frac{n\delta^2}{2+\delta^2}\,\mathfrak{Q}_{n-1/2}(\chi)  \;\lesssim\; n \cdot \frac{1}{n^3\delta^3} \;=\; \frac{1}{n^2\delta^3}.
\end{equation*}
Using \eqref{eq:q22}, we obtain that for $\delta>0$,
\begin{equation*}
 \frac{n\delta^2}{2+\delta^2}\mathfrak{Q}_{n-1/2}(\chi)  \;\leq\;   \frac{n\delta^2}{2+\delta^2} (\frac{1}{n^3\delta^2} + \frac{1}{n^3\delta^3})
\;\lesssim\; \frac{1}{n^2\delta^2} + \frac{1}{n^2\delta}.
\end{equation*}
Thus for $\delta < 2$,
\begin{equation*}
\frac{n\delta^2}{2+\delta^2}\mathfrak{Q}_{n-1/2}(\chi) \;\lesssim\; \frac{1}{n^2\delta^2}\,.
\end{equation*}
Using \eqref{eq:q3}, we obtain that for $\delta < \frac{1}{2n}$,
\begin{equation*}
\frac{n\delta^2}{2+\delta^2}\mathfrak{Q}_{n-1/2}(\chi) \;=\; \frac{n^2\delta^2}{(2+\delta^2)n}\mathfrak{Q}_{n-1/2}(\chi) \;\leq\; \frac{1}{8n}\mathfrak{Q}_{n-1/2}(\chi) \;\lesssim\; \log \frac{1}{n\delta}\,.
\end{equation*}

Therefore we are left to consider $ n \left(\mathfrak{Q}_{n-1/2}(\chi)-\mathfrak{Q}_{n-3/2}(\chi) \right)$.
From \eqref{eq:q1}, for $\delta \geq \cone$,
\begin{equation*}
n \left|\mathfrak{Q}_{n-1/2}(\chi)-\mathfrak{Q}_{n-3/2}(\chi) \right| \;\leq\; n\cdot\frac{1}{n^3\delta^3} \;\lesssim\; \frac{1}{n^2\delta^3}.
\end{equation*}
For the other two bounds, we use the explicit expression
\begin{multline*}
2\sqrt{2}n(\mathfrak{Q}_{n-1/2}-\mathfrak{Q}_{n-3/2})(1+\delta^2) = n\int_{-\pi}^\pi \frac{e^{in\theta}(1-e^{-i\theta}) \, d\theta}{\sqrt{\delta^2 + (1-\cos(\theta))}} \\ = -\int_{-\pi}^\pi \frac{e^{i(n-1)\theta} \, d\theta}{\sqrt{\delta^2 + (1-\cos(\theta))}} - \frac{i}{2}\int_{-\pi}^\pi \frac{e^{in\theta}(1-e^{-i\theta}) \sin(\theta) \, d\theta}{(\delta^2 + (1-\cos(\theta)))^{3/2}} \\ 
= - \int_{-\pi}^\pi \frac{\cos((n-1)\theta) \, d\theta}{\sqrt{\delta^2 + (1-\cos(\theta))}} + \frac{1}{2}\int_{-\pi}^\pi \frac{\cos(n\theta) \sin^2(\theta) \, d\theta}{(\delta^2 + (1-\cos(\theta)))^{3/2}} + \frac{1}{2}\int_{-\pi}^\pi \frac{\sin(n\theta)(1-\cos(\theta)) \sin(\theta) \, d\theta}{(\delta^2 + (1-\cos(\theta)))^{3/2}}\,.
\end{multline*}
The first term on the last line coincides with $2\sqrt{2}\mathfrak{Q}_{n-3/2}$, which we have already handled. For the second term,
\begin{multline*}
\int_{-\pi}^\pi \frac{\cos(n\theta)\sin^2(\theta) \, d\theta}{(\delta^2 + (1-\cos(\theta)))^{3/2}} \;=\; \frac{1}{n}\int_{-\pi n}^{\pi n} \frac{\cos(y) \sin^2(y/n) \, dy}{(\delta^2 + (1-\cos(y/n)))^{3/2}} \\
=\; - \frac{2}{n^2}\int_{-\pi n}^{\pi n} \frac{\sin(y) \cos(y/n) \sin(y/n) \, dy}{(\delta^2 + (1-\cos(y/n)))^{3/2}} + \frac{3}{2}\frac{1}{n^2}\int_{-\pi n}^{\pi n} \frac{\sin(y)\sin^3(y/n) \, dy}{(\delta^2 + (1-\cos(y/n)))^{5/2}}.
\end{multline*}
Arguing as before, for $n\delta < 1/2$,
$$
\left|\int_{-\pi}^\pi \frac{\cos(n\theta)\sin^2(\theta) \, d\theta}{(\delta^2 + (1-\cos(\theta)))^{3/2}}\right| \;\lesssim\; \int_{-1}^1 \frac{y^2 }{((n\delta)^2 + y^2 )^{3/2}} + \frac{y^4}{((n\delta)^2 + y^2)^{5/2}} \, dy + \int_{|y| > 1} \frac{1}{y^2} \, dy \;\approx\; \log \frac{1}{n\delta}.
$$
For $\delta> 0$, integrating by parts one more time yields
\begin{align*}
\left|\int_{-\pi}^\pi \frac{\cos(n\theta)\sin^2(\theta) \, d\theta}{(\delta^2 + (1-\cos(\theta)))^{3/2}}\right| &\;\lesssim\;  \int_{-\infty}^{\infty} \frac{ dy}{((n\delta)^2 + y^2)^{3/2}}+  \int_{-\infty}^{\infty} \frac{y^2 dy}{((n\delta)^2 + y^2)^{5/2}} +  \int_{-\infty}^{\infty} \frac{y^4 dy}{((n\delta)^2 + y^2)^{7/2}} \\ &\;\approx\; \frac{1}{(n\delta)^2}.
\end{align*}
The final term is dealt with in an identical manner. 
\end{proof}

\subsection{Proof of Lemmas~\ref{lem:essest} and \ref{lem:localest}} \label{app:localest}

With the estimates of Lemma~\ref{lem:Qasymp} in hand, we can now provide a proof of Lemma~\ref{lem:essest}.
\begin{proof}[Proof of Lemma~\ref{lem:essest}]
	Let $W_\alpha$ be the infinite straight cone obtained as the surface of revolution of the half-line given by $w_\alpha(t) = (\sin(\alpha) t, \cos (\alpha) t)$, $t > 0$. As in Section~\ref{sec:background}, for $n \in \mathbb{Z}$, let $K^n_{w_\alpha}$ denote the corresponding modal operator with kernel 
	$$K^n_{w_\alpha}(t,t') = \frac{1}{2\pi}\int_0^{2\pi} e^{-i n\theta} K_{W_\alpha}(t,\theta, t',0) \, d\theta, \quad t, t' > 0.$$
	Then the holomorphic function
	$\Pi^n_\alpha(z)$ of \eqref{eq:Pidef} is a Mellin transform,
	\begin{equation} \label{eq:mellin}
	\Pi^n_\alpha(z) = \sin \alpha \int_0^\infty t^z K_{w_\alpha}^n(t, 1) \, \frac{dt}{t}, \qquad 0 < \mre z < 3,
	\end{equation}
	see \cite[Section~3]{HelsingPerfekt2018a}. In particular, $t^{\beta-1} K^n_{w_\alpha}(t,1) \in L^1((0,\infty))$ for $0 < \beta < 3$.
	
	If $0 < \alpha < \pi/2$, then the convexity of $W_\alpha$ ensures that $K_{W_\alpha}$ is a non-negative kernel, cf. \eqref{eq:kerK}. In turn, $K^0_{w_\alpha}$ is non-negative, and therefore, by \eqref{eq:mellin},
	$$|\sigma_{0, \alpha}| = \sup_{\xi \in \mathbb{R}} \left|\Pi^0_\alpha(3/2+i\xi)\right| = \Pi^0_\alpha(3/2).$$
	Furthermore, for any $\xi \in \mathbb{R}$ and $n \in \mathbb{Z}$,
	$$ |\Pi^n_\alpha(3/2 + i\xi)|  \leq \sin \alpha \int_0^\infty t^{3/2} |K_{w_\alpha}^n(t, 1)| \, \frac{dt}{t} \leq \Pi^0_\alpha(3/2),$$
	so that $|\sigma_{n, \alpha}| \leq |\sigma_{0, \alpha}|$. On the other hand, the non-negativity of $\sigma_{\ess}(K_\Gamma, \mathcal{E})$ and the continuity of $\Pi^n_\alpha$ guarantees that
	$$\sigma_{\ess}(K_\gamma^n, \mathcal{E}_n) = [0, |\sigma_{n, \alpha}|].$$
	We conclude that
	$$\sigma_{\ess}(K_\Gamma, \mathcal{E}) = [0, |\sigma_{0, \alpha}|].$$
	A similar argument applies when $\pi/2 < \alpha < \pi$.
	
	The parity identity $P^{|n|}_z = P^{|n|}_{-z-1}$ for associated Legendre functions of the first kind implies that $\Pi^n_\alpha(z) = \Pi^n_\alpha(3-z)$, and in particular that
	$$|\tilde{\sigma}_{n, \alpha}| = \sup_{\xi \in \mathbb{R}} \left|\Pi^n_\alpha(1+i\xi)\right| = \sup_{\xi \in \mathbb{R}} \left|\Pi^n_\alpha(2+i\xi)\right|.$$ 
	By the Hadamard three lines theorem we thus have that
	$$|\sigma_{n, \alpha}| \leq |\tilde{\sigma}_{n, \alpha}|.$$
	
	Finally, in the special case of a straight cone, formula \eqref{eq:modalformula} takes for $n \geq 0$ the form
	$$K^n_{w_\alpha}(t, 1) = 	\frac{\cos \alpha}{2\sqrt{2\pi^3} \sin^2 \alpha} \frac{\mathfrak{Q}_{n-1/2}(\chi)+\mathfrak{R}_n(\chi)}{t^{3/2}},$$
	where $\chi = \chi(t) = 1 + \frac{(t - 1)^2}{2 t \sin^2 \alpha}$. Let $$\mathfrak{P}_n(\chi) = (\mathfrak{Q}_{n-1/2} + |\mathfrak{R}_{n}|)(\chi).$$
	Suppose that $c_0 \in (0, \pi/2)$ is given and that $|\alpha - \pi /2| < c_0$. Then $\sin \alpha$ is uniformly bounded from below (depending on $c_0$) and to prove \eqref{eq:bdess} it is sufficient to show that
	$$\int_0^\infty t^{-\frac{1}{2}} \mathfrak{P}_n(\chi(t)) \, \frac{dt}{t} \lesssim \frac{1}{n}, \qquad n \geq 1.$$	
	Making the change of variable $s^2 = \frac{(t - 1)^2}{2 t \sin^2 \alpha}$ for $t < 1$ and for $t > 1$, we find that it is equivalent to show that
	$$\int_0^\infty \mathfrak{P}_n(1 + s^2) \, ds \lesssim \frac{1}{n}.$$	
	By Lemma~\ref{lem:Qasymp},
	$$\int_1^\infty \mathfrak{P}_n(1 + s^2) \, ds \lesssim \frac{1}{n^2}\int_{1}^\infty \frac{ds}{s^3} \lesssim \frac{1}{n^2},$$
	$$ \int_{1/(2n)}^1 \mathfrak{P}_n(1+s^2) \, ds \lesssim \frac{1}{n^2}\int_{1/(2n)}^1 \frac{1}{s^2} \, ds  \approx \frac{1}{n},$$	
	and
	$$ \int_{0}^{1/(2n)} \mathfrak{P}_n(1+s^2) \, ds \lesssim \int_{0}^{1/(2n)} \log \frac{1}{ns} \, ds  \approx \frac{1}{n}. $$
	The case when $n < 0$ is handled by symmetry.
\end{proof}

Lemma~\ref{lem:localest} follows from similar estimates, applying the Schur test.

\begin{proof}[Proof of Lemma~\ref{lem:localest}]
We prove the first inequality; the second inequality can be proved in a nearly identical fashion. We want to show that 
$$\int_0^1 \left| \int_0^1 \chi_1(t) \mathfrak{K}^n(t,t') g (t') d\mu_0(t') \right|^2 d\mu_0(t) \;\lesssim\;  \left( \frac{|\alpha - \pi/2| + \varepsilon}{n}\right)^2 \int_0^1| g(t)|^2 d\mu_0(t), \quad g\in L^2(d\mu_0),$$
uniformly for all $(\alpha,\varepsilon)$-perturbations $\Gamma$ of $\Gamma_0$.

Let 
$$M(t) \;=\; t\chi_{[0,\frac{1}{5}]}(t) + \chi_{[\frac{1}{5},\frac{4}{5}]}(t) + (1-t)\chi_{[\frac{4}{5},1]}(t).$$ 
It is equivalent to show that 
\begin{equation*}
\int_0^1 \left| \int_0^1  R(t,t') f (t') d\mu_1(t') \right|^2 d\mu_1(t) \;\lesssim\; \left( \frac{|\alpha - \pi/2| + \varepsilon}{n}\right)^2  \int_0^1|f(t)|^2 d\mu_1(t), \quad f \in L^2(d\mu_1),
\end{equation*}
where
\begin{align*}
d\mu_1(t) &= \big(M(t)\big)^{-2} d\mu_0(t), \\
 R(t,t') &=  \chi_1(t) \mathfrak{K}^n(t,t') M(t) M(t').
 \end{align*}
By the Schur test, it suffices to verify that
\begin{equation*}
\sup_{t\in[0,1]}\int_0^1 |R(t,t')| d\mu_1(t') \;\lesssim\;  \frac{|\alpha - \pi/2| + \varepsilon}{n}, \quad \sup_{t'\in[0,1]} \int_0^1 |R(t,t')| d\mu_1(t) \;\lesssim\;  \frac{|\alpha - \pi/2| + \varepsilon}{n}.
\end{equation*}
Recalling from \eqref{eq:kernelKS} that
\begin{equation*}
\mathfrak{K}^n (t,t') \;=\;  K_{\gamma}^n(t,t') \frac{\gamma_1(t')|\gamma'(t')|} {\gamma_{0,1}(t')|\gamma_0'(t')|}  - K_{\gamma_0}^n(t,t'),
\end{equation*}
we consider
$$
R_1(t,t') \;=\; \chi_1(t) K_{\gamma}^n(t,t') \frac{\gamma_1(t')|\gamma'(t')|} {\gamma_{0,1}(t')|\gamma_0'(t')|} M(t) M(t').
$$
The other term can be handled similarly.
Let $\hat\chi_1 + \hat\chi_2 =1$ be a partition of unity of $[0,1]$ such that $\mathrm{supp}\,\hat\chi_1\subset[0,\frac{1}{5}]$ and $\mathrm{supp}\,\hat\chi_2\subset[\frac{1}{6},1]$.
It suffices to show that
\begin{equation}\label{eq:4ints}
\begin{aligned}
\sup_{t\in[0,4\varepsilon]}\int_0^1 |R_1(t,t') \hat\chi_1(t')| d\mu_1(t') \;\lesssim\;  \frac{|\alpha - \pi/2| + \varepsilon}{n},\\
\sup_{t'\in[0,\frac{1}{5}]} \int_0^1 |R_1(t,t') \hat\chi_1(t')| d\mu_1(t) \;\lesssim\;  \frac{|\alpha - \pi/2| + \varepsilon}{n},\\
\sup_{t\in[0,4\varepsilon]}\int_0^1 |R_1(t,t') \hat\chi_2(t')| d\mu_1(t') \;\lesssim\;  \frac{|\alpha - \pi/2| + \varepsilon}{n},\\
\sup_{t'\in[\frac{1}{6},1]} \int_0^1 |R_1(t,t') \hat\chi_2(t')| d\mu_1(t) \;\lesssim\;  \frac{|\alpha - \pi/2| + \varepsilon}{n}\,.
\end{aligned}
\end{equation}

By Definition~\ref{def:pert}, we have that
\begin{equation*}
 \frac{\gamma_1(t')|\gamma'(t')|} {\gamma_{0,1}(t')|\gamma_0'(t')|} \approx 1, \quad t \in [0,1].
\end{equation*}
Furthermore, there is a universal constant $C_1$ such that
$$|\gamma(t)-\gamma(t')||K_{\Gamma}(t,0,t',0)| < C_1, \quad t, t' \in [0,1].$$
The assumptions of Definition~\ref{def:pert} ensure that $|\gamma(t) - \gamma'(t)| \approx |t-t'|$. The estimate is thus a consequence of the uniform bounds on $|\gamma|$, $|\gamma'|$ and the bound on the curvature of the generating curve $\gamma$ of $\Gamma$. Also note that
$$
d\mu_1(t) \;\approx\; dt/t \;\text{ on }\, [0,\textstyle\frac{1}{5}], \quad  d\mu_1(t) \;\approx\; dt/(1-t) \;\text{ on }\, [\textstyle\frac{1}{6},1].
$$

Applying these estimates together with \eqref{eq:modalformula} yields that the four integrals in \eqref{eq:4ints} can be bounded as follows:
\begin{equation}\label{eq:R11}
\int_0^1 |R_1(t,t') \hat\chi_1(t')| d\mu_1(t') \;\lesssim\; 
\chi_1(t) \int_0^{\frac{1}{5}} \sqrt{\frac{t'}{ t}} (|\cot\alpha| + \epsilon) \left(\mathfrak{Q}_{n-1/2}(\chi)+  |\mathfrak{R}_n(\chi)|\right) \, \frac{dt'}{t'}\,,
\end{equation}
\begin{equation}\label{eq:R12}
\int_0^1 |R_1(t,t') \hat\chi_1(t')| d\mu_1(t) \;\lesssim\; 
\hat\chi_1(t') \int_0^{4\varepsilon} \sqrt{\frac{t'}{ t}} (|\cot\alpha| + \epsilon) \left(\mathfrak{Q}_{n-1/2}(\chi)+  |\mathfrak{R}_n(\chi)|\right) \, \frac{dt}{t},
\end{equation}
\begin{equation}\label{eq:R13}
\int_0^1 |R_1(t,t') \hat\chi_2(t')| d\mu_1(t') \;\lesssim\; 
\chi_1(t) \int_{\frac{1}{6}}^1 \frac{\sqrt{1-t'}}{\sqrt{ t}} (|\cot\alpha| + \epsilon) \left(\mathfrak{Q}_{n-1/2}(\chi)+  |\mathfrak{R}_n(\chi)|\right) \, \frac{dt'}{1-t'},
\end{equation}
\begin{equation}\label{eq:R14}
\int_0^1 |R_1(t,t') \hat\chi_2(t')| d\mu_1(t) \;\lesssim\; 
\hat\chi_2(t')\int_0^{4\varepsilon} \frac{\sqrt{1-t'}}{\sqrt{ t}}  (|\cot\alpha| + \epsilon) \left(\mathfrak{Q}_{n-1/2}(\chi)+  |\mathfrak{R}_n(\chi)|\right) \, \frac{dt}{t},
\end{equation}
where
$$\chi \;=\; 1+ \frac{|\gamma(t)-\gamma(t')|^2}{2\gamma_1(t)\gamma_1(t')}.$$

We first consider \eqref{eq:R11}. Note that there is a universal constant $C_1 > 1$ such that
\begin{equation*}
 \frac{|t-t'|^2}{2 t t'}\;\leq\; \frac{|\gamma(t)-\gamma(t')|^2}{2\gamma_1(t)\gamma_1(t')}\;\leq\;  C_1\frac{|t-t'|^2}{2 t t'}
\end{equation*}
whenever $t\in \mathrm{supp}\, \chi_1 \subset [0,4\varepsilon]$ and $t' \in \mathrm{supp}\,\hat\chi_1\subset[0,\frac{1}{5}]$. For $n \geq 1$, let
$$\mathfrak{H}_n(1+\delta^2) =  \sup_{1 \leq \kappa \leq C_1 }\left(\mathfrak{Q}_{n-1/2}(1+\kappa\delta^2)+  |\mathfrak{R}_n(1+\kappa\delta^2)|\right), \quad \delta > 0,$$
so that
$$\mathfrak{Q}_{n-1/2}(\chi(t))+  |\mathfrak{R}_n(\chi(t))| \leq \mathfrak{H}_n(1+s^2), \quad t\in \mathrm{supp}\, \chi_1, \, t' \in \mathrm{supp}\,\hat\chi_1,$$
where 
\begin{equation} \label{eq:schange}
s^2 = \frac{|t-t'|^2}{2 t t'} = \frac{|(t'/t)-1|^2}{2t'/t}.
\end{equation}
For fixed $t$, we understand \eqref{eq:schange} as two changes of variable from $\tilde{t} = t'/t$ to $s^2$, one for $\tilde{t} > 1$ and one for $\tilde{t} < 1$,
\begin{equation*}
ds \;=\; \pm \frac{ t'+ t }{\sqrt{8tt'}} \, \frac{dt'}{t'} \;=\; \pm \frac{ \tilde{t}+ 1 }{\sqrt{8\tilde{t}}} \, \frac{d\tilde{t}}{\tilde{t}}.
\end{equation*}
Making the changes of variable in \eqref{eq:R11} yields, for $t \in \mathrm{supp}\, \chi_1$,
$$\int_0^1 |R_1(t,t') \hat\chi_1(t')| d\mu_1(t')  \;\lesssim\;  (|\cot\alpha| + \epsilon) \int_0^\infty \mathfrak{H}_n(1+s^2) \, ds.$$
Note that Lemma~\ref{lem:Qasymp} immediately implies that $\mathfrak{H}_n$ satisfies the same estimates as $\mathfrak{Q}_{n-1/2}$. Hence the same argument as in the proof of Lemma~\ref{lem:essest} shows that $\int_0^\infty \mathfrak{H}_n(1+s^2) \, ds \lesssim 1/n$. The integral of \eqref{eq:R12} is treated in the same way.

Finally we consider \eqref{eq:R13} and \eqref{eq:R14}.
In this case $t\in \mathrm{supp}\, \chi_1 \subset [0,4\varepsilon]$ and $t' \in \mathrm{supp}\,\tilde\chi_1\subset[\frac{1}{6},1]$, and therefore 
$$
\frac{|\gamma(t)-\gamma(t')|^2}{2\gamma_1(t)\gamma_1(t')} \;\gtrsim\; \frac{1}{t (1-t')}\,.
$$
For \eqref{eq:R13}, we find from Lemma~\ref{lem:Qasymp} that
\begin{equation*}
\sup_{t\in[0,4\varepsilon]} \int_0^1 |R_1(t,t') \hat\chi_2(t')| d\mu_1(t')  \;\lesssim\; 
\frac{|\cot\alpha| + \epsilon}{n^2} \sup_{t\in[0,4\varepsilon]} t \int_{\frac{1}{6}}^1 (1-t') \, dt'
  \;\lesssim\; \frac{\varepsilon  (|\cot\alpha| + \epsilon)}{n^2},
\end{equation*}
and for \eqref{eq:R14}, 
\begin{equation*}
\sup_{t'\in[\frac{1}{6},1]} \int_0^2 |R_1(t,t') \hat\chi_2(t')| d\mu_1(t) \;\lesssim\;
\frac{|\cot\alpha| + \epsilon}{n^2} \int_0^{4\varepsilon} \, dt
\;\lesssim\;  \frac{\varepsilon(|\cot\alpha| + \epsilon)}{n^2}. \qedhere
\end{equation*}
\end{proof}


\subsection{Proof of Lemma~\ref{lem:Schatten}} \label{app:Schatten}
\begin{proof}
It suffices to prove that $\chi_2 \mathfrak{K} \tilde{\chi}_1,\chi_2 \mathfrak{S} \tilde{\chi}_1 \in H^{\mu}(\Gamma_0 \times \Gamma_0)$ for some $\mu>1$. Because then, by \cite[Theorem 1]{DelgadoRuzhansky2014}, we have that $\chi_2 \mathfrak{K} \tilde{\chi}_1 \in S_p(L^2(\mathbb S^2))$ for every  $p> \frac{2n}{n+2\mu}$, where $n=2$ is the dimension of $\Gamma_0 = \mathbb S^2$.  

Since $\chi_2$ and $\tilde{\chi}_1$ are supported away from the south pole and the north pole, respectively, we switch to an Euclidean coordinate chart $(x,y)$ for points in the support of $\chi_2$ and a chart $(u,v)$ for points in the support of $\tilde{\chi}_1$. More precisely, we let $(x,y)$ be given by stereographic projection from the south pole, and $(u,v)$ by projection onto the horizontal plane. 

As the supports of $\chi_2$ and $\tilde{\chi}_1$ do not overlap, the functions $$\chi_2(t(x,y)) K_{\Gamma_0}(\br_0(x,y), \br_0(u,v)) \tilde{\chi}_1(t(u,v)) \quad \textrm{and} \quad \chi_2(t(x,y))S_{\Gamma_0}(\br_0(x,y),\br_0(u,v)) \tilde{\chi}_1(t(u,v))$$ are then smooth on $\mathbb{R}^4$. It therefore suffices to show that the following two functions are in $H^{\mu}(\mathbb R^4)$ for some $\mu>1$:
\begin{equation*}
\begin{aligned}
\chi_2(t(x,y)) K_{\Gamma}(\br(x,y),\br(u,v))  \frac{ \gamma_1(t(u,v)) |\gamma'(t(u,v))|} { \gamma_{0,1}(t(u,v)) |\gamma_0'(t(u,v))|}  \tilde{\chi}_1(t(u,v)),\\
\chi_2(t(x,y)) S_{\Gamma}(\br(x,y),\br(u,v))  \frac{ \gamma_1(t(u,v)) |\gamma'(t(u,v))|}  {\gamma_{0,1}(t(u,v)) |\gamma_0'(t(u,v))|}  \tilde{\chi}_1(t(u,v)).
\end{aligned}
\end{equation*}
Using that $\gamma(t) = (t,\gamma_2(t))$ and $\gamma_0(t) = (t,\gamma_{0,2}(t))$ for $t \in \mathrm{supp}\,{\tilde{\chi}_1} \subset [0,2\varepsilon]$ these two functions can be rewritten as
\begin{equation}\label{eq:estK}
\begin{aligned}
\frac{1}{2\pi} \chi_2(t(x,y))  \frac{\langle \br(x,y) - \br(u,v), \bnu_{\br(x,y)} \rangle}{|\br(x,y) - \br(u,v)|^3}  \frac{ \sqrt{ 1 + \gamma_2'(\sqrt{u^2+v^2})^2}} { \sqrt{ 1 + \gamma_{0,2}'(\sqrt{u^2+v^2})^2}}  \tilde{\chi}_1(\sqrt{u^2+v^2}),\\
\frac{1}{2\pi}  \chi_2(t(x,y)) \frac{1}{|\br(x,y) - \br(u,v)|} \frac{ \sqrt{ 1 + \gamma_2'(\sqrt{u^2+v^2})^2}} { \sqrt{ 1 + \gamma_{0,2}'(\sqrt{u^2+v^2})^2}} \tilde{\chi}_1(\sqrt{u^2+v^2}).
\end{aligned}
\end{equation}
Away from $(u,v)=(0,0)$, both expressions are smooth in all variables.
In a neighborhood of $(u,v) = (0,0)$, $ \br(u,v) = (u,v, \cot(\alpha) \sqrt{u^2+v^2}+ c)$, $\gamma_2'(\sqrt{u^2+v^2}) = \cot\alpha$, and $\gamma_{0,2}'(\sqrt{u^2+v^2} ) = \frac {\sqrt{u^2+v^2}}{\sqrt{1-u^2-v^2}}$.

Both functions in \eqref{eq:estK} are therefore of the form
	\begin{equation*}
	F(x,y, u,v) = G(x,y, u,v, \sqrt{u^2+v^2}),
	\end{equation*}
	where 
	$G(x,y,u,v,w)$ is a smooth compactly supported function in $\mathbb R^5$. Let 
	$$G(x,y,u,v,w) = G_0(x,y,u,v) + G_1(x,y,u,v)w + G_2(x,y,u,v)w^2 + G_3(x,y,u,v, w)$$
	be the Taylor expansion of $G$ around $w = 0$, where $G_0, \ldots, G_3$ are smooth and 
	$$|\partial_{x,y,u,v}^\alpha \partial_w^j G_3(x,y,u,v,w)| \leq C_{\alpha, j} |w|^{3-j}$$ for $j=0,1,2$, any multi-index $\alpha \in \mathbb{N}_0^4$, and some corresponding constants $C_{\alpha, j}$.
 It is then clear that 
 $$F(x,y,u,v) - G_1(x,y,u,v)\sqrt{u^2+v^2} \in H^2(\mathbb{R}^4).$$
 
	To finish the proof, we therefore only need to check that $\tilde{G}(x,y,u,v) \sqrt{u^2+v^2} \in H^\mu( \mathbb R^4)$ for $\mu < 2$, where $\tilde{G}$ is a smooth compactly supported function in $\mathbb{R}^4$. By comparing this function with $\tilde{G}(x,y,u,v)(1-e^{-\sqrt{u^2+v^2}})$, it is enough to note that $e^{-\sqrt{u^2+v^2}} \in H^{\mu}(\mathbb{R}^2)$ for $\mu < 2$, as can be seen from an explicit computation of its Fourier transform.
\end{proof}

\bigskip
\noindent{\bfseries Acknowledgement.}
We are grateful to Lucas Chesnel for pointing out the relevance of \cite[Theorem~7.5]{Bonnet-BenDhia2012}, and to Titus Hilberdink for helpful discussions around Lemma~\ref{lem:Qasymp}.
This material is based upon work supported by the National Science Foundation under Grant No. DMS-1814902 (S.P. Shipman). The research of K.-M. Perfekt was supported by grant EP/S029486/1 of the UK Engineering and Physical Sciences Research Council.

\bibliographystyle{amsplain}


\begin{thebibliography}{10}

\bibitem{BonnetierZhang2017}
Eric Bonnetier and Hai Zhang, \emph{Characterization of the essential spectrum
  of the {N}eumann-{P}oincar\'{e} operator in 2{D} domains with corner via
  {W}eyl sequences}, Rev. Mat. Iberoam. \textbf{35} (2019), no.~3, 925--948.

\bibitem{LCP21}
Marta de~Le\'on-Contreras and Karl-Mikael Perfekt, \emph{The quasi-static
  plasmonic problem for polyhedra}, arXiv:2103.13071 (2021).

\bibitem{DelgadoRuzhansky2014}
Julio Delgado and Michael Ruzhansky, \emph{Schatten classes on compact
  manifolds: kernel conditions}, J. Funct. Anal. \textbf{267} (2014), no.~3,
  772--798.

\bibitem{Bonnet-BenDhia2012}
Anne-Sophie Bonnet-Ben Dhia, Lucas Chesnel, and Patrick~Ciarlet Jr.,
  \emph{T-coercivity for scalar interface problems between dielectrics and
  metamaterials}, ESAIM: Mathematical Modelling and Numerical Analysis
  \textbf{46} (2012), no.~6, 1363--1387.

\bibitem{Bonnet-BenHazardMonteghett2020}
Anne-Sophie Bonnet-Ben Dhia, Christophe Hazard, and Florian Monteghetti,
  \emph{Complex-scaling method for the plasmonic resonances of planar
  subwavelength particles with corners}, preprint, hal-02923259 archives
  ouvertes (2020).

\bibitem{HelsingKangLim2016}
Johan Helsing, Hyeonbae Kang, and Mikyoung Lim, \emph{Classification of spectra
  of the {N}eumann-{P}oincar{\'e} operator on planar domains with corners by
  resonance}, Annales de l'Institut Henri Poincare (C) Nonlinear Analysis
  \textbf{34} (2017), no.~4, 991--1011.

\bibitem{HelsingPerfekt2018a}
Johan Helsing and Karl-Mikael Perfekt, \emph{The spectra of harmonic layer
  potential operators on domains with rotationally symmetric conical points},
  J. Math. Pures Appl. (9) \textbf{118} (2018), 235--287.

\bibitem{KangLimYu2017}
Hyeonbae Kang, Mikyoung Lim, and Sanghyeon Yu, \emph{Spectral resolution of the
  {N}eumann-{P}oincar{\'e} operator on intersecting disks and analysis of
  plasmon resonance}, Archive for Rational Mechanics and Analysis \textbf{226}
  (2017), no.~1, 83--115.

\bibitem{KhavinsonPutinarShapiro2007}
Dmitry Khavinson, Mihai Putinar, and Harold~S. Shapiro, \emph{Poincar{\'e}'s
  variational problem in potential theory}, Archive for Rational Mechanics and
  Analysis \textbf{185} (2007), no.~1, 143--184.

\bibitem{LiShipman2019}
Wei Li and Stephen~P. Shipman, \emph{Embedded eigenvalues for the
  {N}eumann-{P}oincar{\'e} operator}, J. Integral Equations and Appl.
  \textbf{31} (2019), no.~4, 505--534.

\bibitem{Med97}
Dagmar Medkov\'{a}, \emph{The third boundary value problem in potential theory
  for domains with a piecewise smooth boundary}, Czechoslovak Math. J.
  \textbf{47(122)} (1997), no.~4, 651--679.

\bibitem{Perfekt2020}
Karl-Mikael Perfekt, \emph{Plasmonic eigenvalue problem for corners: Limiting
  absorption principle and absolute continuity in the essential spectrum}, J.
  Math. Pures Appl. (9) (2020), in press.

\bibitem{PerfektPutinar2014}
Karl-Mikael Perfekt and Mihai Putinar, \emph{Spectral bounds for the
  {Neumann-Poincar{\'e}} operator on planar domains with corners}, Journal
  d'Analyse Math{\'e}matique \textbf{124} (2014), no.~1, 39--57.

\bibitem{PerfektPutinar2017}
\bysame, \emph{The essential spectrum of the {Neumann-Poincar{\'e}} operator on
  a domain with corners}, Archive for Rational Mechanics and Analysis
  \textbf{223} (2017), no.~2, 1019--1033.

\end{thebibliography}

\providecommand{\bysame}{\leavevmode\hbox to3em{\hrulefill}\thinspace}
\providecommand{\MR}{\relax\ifhmode\unskip\space\fi MR }
\providecommand{\MRhref}[2]{%
  \href{http://www.ams.org/mathscinet-getitem?mr=#1}{#2}
}
\providecommand{\href}[2]{#2}

\end{document}